\documentclass[12pt,makeidx]{amsart}
\usepackage[foot]{amsaddr}

\usepackage{enumerate}
\usepackage{amssymb,mathrsfs, amsmath, amsthm, bm}
\usepackage{soul}
\usepackage[usenames,dvipsnames]{xcolor}
\definecolor{darkblue}{rgb}{0.0, 0.0, 0.55}

\usepackage[pagebackref,colorlinks,linkcolor=BrickRed,citecolor=OliveGreen,urlcolor=blue,hypertexnames=true]{hyperref}

\newtheorem{theorem}{Theorem}[section]
\newtheorem{cor}[theorem]{Corollary}
\newtheorem{lemma}[theorem]{Lemma}

\newtheorem{proposition}[theorem]{Proposition}
\newtheorem{definition}[theorem]{Definition}

\newtheorem{remark}[theorem]{Remark}

\newtheorem{Cor}[theorem]{Corollary}

\newcommand{\fp}{\bm p}
\def\Vfp{\mathfrak V}
\def\Vfpj{\mathfrak V(\fp_j)}

\def\cQ{\mathcal Q}
\def\cP{\mathcal P}

\def\C{\mathbb C}
\def\cH{ H}
\def\cD{\mathcal D}
\def\cK{ K}
\def\K{\mathcal K}
\def\cB{\mathcal B}
\def\cX{\mathcal X}
\def\cW{\mathcal W}
\def\fps{\fp_1\,\fp_2\, \cdots \fp_s}

\def\cL{L}
\def\cF{ F}

\def\N{\mathbb N}

\DeclareMathOperator{\bimult}{bimult}
\DeclareMathOperator{\mult}{mult}
\DeclareMathOperator{\rank}{rank}
\newcommand{\df}[1]{{\it{#1}}{\index{#1}}}

\makeindex

\usepackage{upgreek}
\usepackage{cmap}

\title{Finite Rank Isopairs}

\author[Udeni Wijesooriya]{Udeni Wijesooriya}
\address{Department of Mathematics\\
  University of Florida\\ Gainesville}
  \email{wudeni.pera06@ufl.edu}

\begin{document}


\maketitle

 
 \begin{abstract} An algebraic isopair is a commuting pair of  pure isometries  that is annihilated by a  polynomial.  The notion of the rank of a pure algebraic isopair with finite bimultiplicity is introduced as an $s$-tuple $\alpha=(\alpha_1,...,\alpha_s)$ of natural numbers. A  pure algebraic isopair of finite bimultiplicity with rank $\alpha$, acting on a Hilbert space is nearly  $\max\{\alpha_1,...,\alpha_s\}$-cyclic and there is a finite codimensional invariant subspace such that the restriction to that subspace  is  $\max\{\alpha_1,...,\alpha_s\}$-cyclic. 
 
\end{abstract}
 
\section{INTRODUCTION} 
 Given a polynomial $p\in \C[z,w]$ (or in $\C[z]$) let $Z(p)$ denote its zero set. \index{$Z(p)$} We say $p$ is \df{square free} if $q^2$ does not divide $p$ for every non-constant polynomial $q(z,w) \in \mathbb{C}[z,w]$. We say $q \in \mathbb{C}[z,w]$ is the \df{square free version} of $p$ if $q$ is the polynomial with smallest degree such that  $q$ divides $p$ and $Z(p)=Z(q)$. The square free version is unique up to multiplication by a nonzero constant.

Let \df{$\mathbb D$}, \df{$\mathbb T$} and \df{$\mathbb E$} denote the open unit disk, the boundary of the unit disk and complement of the closed unit disk in $\C$ respectively. In \cite{AKM} the notion of an inner toral polynomial is introduced. (See also \cite{AMS,AM06,Knesetrans,PS}.)  A polynomial $q\in \C[z,w]$ is \df{inner toral} if 
\[
 Z(q)\subset \mathbb D^2  \cup \mathbb T^2 \cup \mathbb E^2.
\]
 In other words, if $(z,w)\in Z(q)$ then either $|z|,|w|<1$ or $|z|=1=|w|$ or $|z|,|w|>1$.  A \df{distinguished variety} in $\C^2$  is the zero set of an inner toral polynomial. 
 
 Let $V$ be an isometry defined on a Hilbert space $\cH$. By the Wold Decomposition, there exist two reducing subspaces for $V$, say $\cK$ and $\cL$, such that $\cH= \cK \oplus \cL$ and $ S=V|_{\cK}$ is a shift operator and $U=V|_{\cL}$ is a unitary operator. We say $V$ is \df{pure}, if there is no unitary part. An isometry $V$ is  pure  if and only if   $\bigcap_{j=1}^{\infty} V^j(\cH) = \{ 0 \}$. A subspace $\cW$ of $\cH$ is called a wandering subspace for $V$ if $V^n(\cW) \perp V^m(\cW)$ for $n \neq m$ and $H=\bigoplus_{n=0}^{\infty} V^n(\cW)$. If $V$ is a pure isometry and $\cW=\cH \ominus V(\cH) =\ker(V^*)$, then $\ker(V^*)$ is a wandering subspace for $V$. Moreover, if $V$ is a pure isometry  then $V \cong M_z$ on the Hilbert Hardy space $H_{\cW}^2$ of $\cW$-valued functions for a Hilbert space $\cW$ with dimension of $\dim(\ker(V^*))$. The \df{multiplicity} of a pure isometry $V$ is defined as $\mult(V)=\dim(\ker(V^*))$.

 A \df{pure isopair} is a pair of commuting pure isometries. A pure isopair  $V=(S,T)$  is a \df{pure algebraic isopair}   if there is a nonzero polynomial $q \in \C[z,w]$ such that $q(S,T)=0$ and  is also referred to as \df{pure $q$-isopair}. The study of pure algebraic isopairs was initiated in \cite{AKM} and also discussed in \cite{M12}. Among the many results in  \cite{AKM} it is shown (see Theorem 1.20) if $V=(S,T)$ is a pure  algebraic isopair, then there is a square free inner toral polynomial $\fp$ such that $\fp(S,T)=0$ that is minimal in the sense if $q(S,T)=0$, then $\fp$ divides $q$. We call this polynomial $\fp$ the \df{minimal polynomial}  of $V$. The minimal polynomial of $V$ is unique up to multiplication by a nonzero constant. Moreover, in \cite{AKM} the notion of a \df{nearly cyclic} pure isopair is introduced. Here we fix a square free inner toral polynomial $\fp$ and consider nearly multi-cyclic pure  isopairs with the  minimal polynomial $\fp$.    
 
 An isopair $V=(S,T)$ acting on a Hilbert space $\cH$ is called  \df{at most nearly $ k$-cyclic} if there exist    distinct $f_1,...,f_{k} \in \cH$ such that the closure of 
 \begin{equation}\label{cyc} \{\sum_{j=1}^k q_j(S,T)f_j  : q_j  \in \mathbb{C}[z,w] \textrm{ for } j=1,2,...,k \}
 \end{equation}
is of finite codimension in $\cH$. It is called  \df{at least nearly $\ k$-cyclic} if the closure of 
$$
\{\sum_{j=1}^ l q_j(S,T)f_j  : q_j  \in \mathbb{C}[z,w] \textrm{ for } j=1,2,...,l \}
$$
is not of finite codimension in $\cH$ for any $l < k$ and for any set of $f_1,...,f_{l} \in \cH$. We say $V=(S,T)$ is \df{nearly $ k$-cyclic} if it is both at most nearly $k$-cyclic and at least nearly $ k$-cyclic. Moreover, $V=(S,T)$ is called \df{$k$-cyclic} if it is nearly $ k$-cyclic and  the span given in $\eqref{cyc}$ is dense in $\cH$.

Given a pair of isometries $V=(S,T)$, define the \df{bimultiplicity} of $V$ by
\[
\bimult(V)=(\mult(S),\mult(T)).
\]

It is a well known fact that we can view pure isopaires as pairs of multiplication operators. In particular, if \  $V=(S,T)$ is a pure $\fp$-isopair of finite multiplicity $(M,N)$, then there exists an  $M \times M$ matrix-valued  rational inner function $\Phi$ with its poles in $\mathbb{E}$, such that $V$ is unitarily equivalent to $(M_z, M_{\Phi})$ on $H^2_{\mathbb{C}^M}$ and $\fp(M_z, M_{\Phi})=0$ (see \cite{AKM}).   Moreover
\begin{equation}\label{it:plambdazero}
\fp(\lambda,\Phi(\lambda))=0 \textrm{ for } \lambda \in \overline{\mathbb{D}}.
\end{equation} 

 \begin{definition}
   \label{def:regular}
We say a point $(\lambda,\mu) \in \mathbb{C}^2$ is a \df{regular point for $\fp$}  if $(\lambda,\mu) \in Z(\fp)$ , but 
\[
\nabla \fp(\lambda,\mu) = \left(\frac{\partial \fp}{\partial z} , \frac{\partial \fp}{\partial w}\right)| _{(\lambda,\mu)} \neq 0.\]
\end{definition}

 Let $\fp$ be a square free inner toral polynomial. Write $\fp=\fp_1\,\fp_2\,...\fp_s$ as a product of (distinct) irreducible factors. Then each $\fp_j$ is inner toral. In other words, each $Z(\fp_j)$ is a distinguished variety. The zero set of $\fp$  is the union of the zero sets of $\fp_j$. Let \index{$\Vfp(\fp)$}
  \[
 \Vfpj =Z(\fp_j)\cap \mathbb D^2, \ \ \ \Vfp(\fp) = Z(\fp)\cap \mathbb D^2 = \displaystyle{\bigcup_{j=1}^s \Vfpj}.
\]
Let $\mathbb{N}$ denote the non negative integers and $\mathbb{N_+}$  denote the positive integers.

\begin{proposition} \label{jointkernelconstant}
	Let $V=(S,T)$ be a pure $\fp$-isopair of finite bimultiplicity  with minimal polynomial $\fp$ and  suppose $\fp=\fp_1\,\fp_2\, \cdots \fp_s$,  a product of distinct irreducible factors. For each $j$ and $(\lambda,\mu)\in \Vfpj$ that is a regular point for $\fp$, the dimension of the intersection of $\ker(S-\lambda)^*$ and $\ker(T-\mu)^* $ is a nonzero constant. 
\end{proposition}

\begin{definition}
Let $V=(S,T)$ be a pure $\fp$-isopair of finite bimultiplicity  with minimal polynomial $\fp$ and  suppose $\fp=\fp_1\,\fp_2\, \cdots \fp_s$,  a product of distinct irreducible factors. The rank of $V$ is a $s$-tuple,  $\alpha=(\alpha_1,\dots,\alpha_s) \in \mathbb{N}_+^s$, denoted by  $\rank(V)$, where 
\[ \alpha_j = \dim \left(\ker(S-\lambda)^*\cap \ker(T-\mu)^*\right),
\]
and $(\lambda,\mu)\in \Vfpj$ and a regular point for $\fp$.
\end{definition}



\begin{theorem} \label{thm:main}
 Suppose $V=(S,T)$ is a pure $\fp$-isopair of finite bimultiplicity   with minimal polynomial $\fp$ and  write $\fp=\fp_1\,\fp_2\, \cdots \fp_s$ as a product of distinct irreducible factors. If $V$ has  rank $(\alpha_1,\alpha_2,...,\alpha_s)$,  then $V$ is   nearly $\max\{\alpha_1,...,\alpha_s\}$-cyclic. 
\end{theorem}

\begin{remark}\rm
 Compare Theorem \ref{thm:main} with the results in \cite{AKM}. 
 
\end{remark}

We prove Theorem \ref{thm:main} in section \ref{sec:5}.  An important ingredient in the proof of Theorem \ref{thm:main} is a representation for a pure $\fp$-
isopair as a pair of multiplication operators on a reproducing kernel Hilbert space over $\Vfp(\fp)$ in the case $\fp$ is irreducible.    Representations of this type already appear in the literature, \cite[Theorem D.14]{DC}   for instance. Here we provide  additional information. See Theorems \ref{thm:KtoST} and \ref{S,T nearly alpha cyclic}.

\begin{remark}
The concept  of nearly multi-cyclic isopairs was introduced in \cite{AKM}.  A discussion on multicyclicity of a bundle shift  given in terms of   its multiplicities  can be found in \cite{AD76}.  In \cite{R69}, the article presents a way to realize a Riemann surface with a   distinguished variety. 
\end{remark}

\section{PRELIMINARIES}

\begin{proposition}
  \label{prop:finite-codim}
  Suppose $p,q\in \C[z,w]$.
\begin{enumerate}[(i)]
 \item
  \label{it:weak-bezout}
     $Z(p)\cap Z(q)$ is a finite set if and only if $p$ and $q$ are relatively prime.
 \item 
   \label{it:finite}
   If $p$ and $q$ are relatively prime, then the ideal $I\subset\C[z,w]$ generated by $p$ and $q$ has finite codimension in $\C[z,w]$; i.e there is a finite dimensional subspace $\mathcal R$  of $\C[z,w]$ such that for each $\psi \in \C[z,w]$ there exist polynomials $s,t \in \C[z,w]$ and $r \in \mathcal R $ such that 
\[
   \psi = sp+tq+ r.
 \]
\end{enumerate}
\end{proposition}

 Bezout's Theorem says that if two algebraic curves, say described by $p=0$ and $q=0$,  do not have any common factors, then they have only finitely many points in common. In particular if $p$ and $q$ do not have any common factors, then $Z(p)$ and $Z(q)$ have only finitely many points in common.  In particular, for the ideal $I$ generated by $p$ and $q$, the affine variety $V(I)= Z(p) \cap Z(q) $ is finite. The  {\it Finiteness Theorem} of \cite[page 13]{DC}, says that if $V(I)$ is finite then the quotient ring $\C[z,w] / I$ has a finite dimension. Hence the ideal $I$ has finite codimension in $\C[z,w]$.

For $p\in \C[z,w]$ and  $\lambda \in\mathbb D$, let $p_\lambda(w)=p(\lambda,w)$   \index{$p^\mu$}.

\begin{lemma}
\label{lem:divide}
 Suppose $\fp$ is square free and inner toral and write $\fp=\fps$  as a product of   irreducible factors. Let $q$ be a nonzero polynomial.
\begin{enumerate}[(i)]
\item 
 \label{it:divide} If $q$ vanishes on a countably  infinite  subset of $\Vfpj$, then $\fp_j$ divides $q$.
\item  \label{it:qvanish} If $q$ vanishes on a cofinite subset of $\Vfp(\fp)$, then $\fp$ divides $q$.
\item \label{it:intersectD} If $Z(q)\cap Z(\fp)\cap \mathbb D^2$ is finite, then $q$ and $\fp$ are relatively prime.
\item \label{it:deriisfinite}The polynomial $\frac{\partial \fp}{\partial w}$ has only finitely many zeros in $\Vfp(\fp)$.
\item \label{it:partialq} If $q \frac{\partial \fp}{\partial w}$ is zero on a cofinite subset of $\Vfp(\fp)$, then $\fp$ divides $q$.
 \item \label{it:Lambda cofinite} If $\Lambda$ is the set of all $\lambda \in \mathbb{D}$ for which $\fp_\lambda (w)$  has distinct zeros, then $\Lambda \subset \mathbb{D}$ is cofinite. 
\end{enumerate}
\end{lemma}

\begin{proof}
	The proof of item \eqref{it:divide} follows from  Proposition \ref{prop:finite-codim} item \eqref{it:weak-bezout} and by the fact that  $\fp_j$ is irreducible.  By item \eqref{it:divide}, each $\fp_j$ divides $q$. Since  the $\fp_j$'s are distinct, their product divides $q$, proving item \eqref{it:qvanish}. If $q$ and $\fp$ have a common factor, then because $\fp$ is inner toral,  $Z(q)$ and $Z(\fp)$ have infinitely many common points in $\mathbb{D}^2$,  proving \eqref{it:intersectD}.

	Let  $q=\frac{\partial \fp}{\partial w}$ and  suppose $q$ has infinitely many zeros in $\Vfp(\fp)$. In this case there is a $j$ such that $q$ has infinitely many zeros in $\Vfp(\fp_j)$.  Hence by \eqref{it:divide}, $q$ vanishes on $\Vfp(\fp_j)$.  
	Therefore, either $\frac{\partial \fp_j}{\partial w}$ has infinitely many zeros in $\Vfp(\fp_j)$ or there is an $\ell$ such that $\fp_\ell$ has infinitely many zeros in $\Vfp(\fp_j)$ and  thus,  by part \eqref{it:divide}, $\fp_j$ divides  $\frac{\partial \fp_j}{\partial w}$ or   $\fp_j$ divides $\fp_\ell$, a contradiction.  
	Item \eqref{it:partialq} follows from item \eqref{it:qvanish}. To prove  item (\ref{it:Lambda cofinite}), if $\Lambda$ is not cofinite, then 
		$\frac{\partial \fp}{\partial w} $  has infinitely many zeros in $Z(\fp)$. Since $\fp$ is inner toral, $\frac{\partial \fp}{\partial w} $  has infinitely many zeros in $\Vfp(\fp)$,  a contradiction to  item \eqref{it:deriisfinite} and hence $\Lambda$ is  cofinite.
\end{proof}

\begin{proposition}
 \label{prop:squarefree}
 Suppose $p \in \C[z,w]$   is a square free polynomial and write $p=p_1\, p_2\, \cdots \, p_s$  as a product of irreducible factors $p_j\in \C[z,w]$.  If $q\in \C[z,w]$ and $Z(p) \subseteq Z(q)$,  then there exist $\gamma=(\gamma_1,...,\gamma_s) \in \N_+ ^s$ and an $r\in \C[z,w]$ such that  $p_j$ and $r$ are relatively prime and
\[
  q = p_1^{\gamma_1}\, p_2^{\gamma_2} \, \cdots \, p_s^{\gamma_s} \, r.
\]
\end{proposition}

\begin{proof}
 The proof is an application of Bezout's Theorem. 
\end{proof}

\begin{remark}
	If $p$ and $q$ are  inner toral polynomial, then we may replace the condition $Z(p)\subseteq Z(q)$ with $\Vfp(\fp) \subseteq Z(q)$. 
\end{remark}

\section{RESULTS FOR GENERAL $\textbf{\em{p}}$}

 In this section $\fp=\fps$ is a general square free inner toral polynomial with (distinct) irreducible factors $\fp_j$.  Let $(n_j,m_j)$ be the bidegree of $\fp_j(z,w)$.
 
In \cite{AKM} it is proven that any nearly cyclic pure $\fp$-isopair is unitarily equivalent to a cyclic pure $\fp$-isopair  restricted to a finite codimensional invariant subspace (See Proposition 3.6 in \cite{AKM}). Next Proposition is a  more generalized version of this result. 

\begin{proposition}
 \label{prop:restrictcyclic}
 Suppose  $V=(S,T)$ is a pure $\fp$-isopair of finite bimultiplicity $(M,N)$  acting on the Hilbert space $\cK$. If $\cH$ is a finite codimension $V$-invariant subspace  of $\cK$ and $W$ is the restriction of $V$ to $\cH$, then there exists a finite codimension subspace $\cL$ of $\cH$ such that $V$ is unitarily equivalent to the restriction of $W$ to $\cL$.  
\end{proposition}

\begin{remark}\rm
 \label{rem:codimsadd}
   In  case the codimension of $\cH$ is one, the codimension of $\cL$ (in $\cH$) can be chosen as $N-1$ (or as $M-1$).  In general, the proof yields a relation between the codimensions of $\cH$ in $\cK$ and $\cL$ in $\cH$ (or in $\cK$).
\end{remark}

\begin{Cor} \label{cor:nearlyuniequi}
 Suppose  $V=(S,T)$ is a pure $\fp$-isopair of finite bimultiplicity $(M,N)$ acting on the Hilbert space $\cK$. If there exists a finite codimension $V$-invariant subspace $\cH$ of $\cK$ such that the restriction of $V$ to $\cH$ is $\beta$-cyclic, then there exists a $\beta$-cyclic pure isopair $W$ acting a Hilbert space $\cL$ and a finite codimension $W$-invariant subspace $\cF$ of $\cL$  such that $W|_\cF$ is unitarily equivalent to $V$.
 \end{Cor}

\begin{proof}[Proof of Theorem \ref{prop:restrictcyclic}]
 Following the argument in \cite[Proposition 3.6]{AKM}, let $\cF = \cK\ominus \cH$ and write, with respect to the decomposition $\cK=\cH\oplus \cF$,
\begin{equation}
 \label{eq:subrep}
 V=(S,T) = \begin{pmatrix} W=(S,T)|_\cH & (X,Y)\\ 0 & (A,B) \end{pmatrix}.
\end{equation}
 In particular $A$ (and likewise $B$) is a contraction on a finite dimensional Hilbert space.  Because $V$ is pure and $A$ is a contraction, $A$ has spectrum in the open disc $\mathbb D$. Choose a (finite) Blaschke $u$ such that $u(A)=0$.  Note that $u(S)$ is an isometry on $\cK$ and moreover the codimension of the range of $u(S)$ (equal to the dimension of the kernel of $u(S)^*$) in $\cK$ is (at most) $dM$, where $d$ is the degree (number of zeros) of $u$. Further, since
\[
 u(S)= \begin{pmatrix} u(S|_\cH) & X^\prime \\ 0 & u(A)=0\end{pmatrix},
\]
 the range $\cL=u(S)\cK$ of $u(S)$ is a subspace of $\cH$ of finite codimension.   Since $u(S)V=Wu(S)$ it follows that $\cL$ is invariant for $W$ and $V$ is unitarily equivalent to $W$ restricted to $\cL$. 

To prove the remark, note that if $A$ is a scalar (equivalently $\cH$ has codimension one in $\cK$), then $u$ can be chosen to be a single Blaschke factor. In which case the codimension of $\cL$ is $N$ in $\cK$ and hence $N-1$ in $\cH$. In general, if $d$ is the degree of the Blaschke $u$, then the codimension of $\cL$ in $\cK$ is $dN$. By reversing the roles of $S$ and $T$ one can replace $N$ with $M$, the multiplicity of the shift $T$.
\end{proof}

\begin{proposition}\label{alphaexist}
		Let $(M_z, M_{\Phi})$ be a pure isopair of finite bimultiplicity $(M,N)$ with minimal polynomial $\fp$, where $\Phi(z)$ is an $M \times M$ matrix-valued rational inner function. There exists an $\alpha  = (\alpha_1,...,\alpha_s) \in \N_+ ^s$  such that 
	\begin{enumerate}[(i)] 
		\item  \label{characteristic poly} for $\lambda\in \mathbb D$, the characteristic polynomial $f_\lambda (w)$ of $\Phi(\lambda)$
		satisfies
		\begin{equation} 
		f_\lambda(w)=\det(w-\Phi(\lambda))= c(\lambda) \fp_{1,\lambda}^{\alpha_1} (w)\, \cdots \fp_{s,\lambda}^{\alpha_s}(w), 
		\end{equation} for a  constant (in $w$)  $c(\lambda)$.
		\item \label{it:diagonalizable} for each $\lambda$ such that $\fp_{\lambda}$ has $m$ distinct zeros, $\Phi(\lambda)$ is diagonalizable and similar to \[ \bigoplus_{j=1}^s \bigoplus_{\mu_j \in Z(\fp_{j,\lambda})} \mu_j I_{\alpha_j},
		\]
		\item  \label{it: dimkernel} if  $(\lambda,\mu) \in Z(\fp_j)$ and  $\frac{\partial \fp}{\partial w}|_{(\lambda,\mu)} \neq 0$, then 
		\[\dim\ker(\Phi(\lambda)-\mu)= \alpha_j.\]
	\end{enumerate} 
\end{proposition}

\begin{proof}
	First note that, by  equation \eqref{it:plambdazero}, for all $\lambda \in \overline{\mathbb D}$ \begin{equation} 
	\label{eq:pmzphi-alttwo}
	\fp_\lambda(\Phi(\lambda))=\fp(\lambda,\Phi(\lambda))=0.
	\end{equation}
	 In particular, the spectrum, $\sigma(\Phi(\lambda))$, is a subset of $Z(\fp_\lambda)$. 
	 
	 Note that $\det(w I_m -\Phi(z))$ is a rational  function whose denominator
	$d(z)$ (a polynomial in $z$ alone)  doesn't vanish in $\overline{\mathbb D}$. Let  $q(z,w)=d(z)\det(wI_m-\Phi(z))$, the numerator of $\det(w I_m -\Phi(z))$. For fixed $z\in \mathbb{D}$, let 
	\[
	q_z (w) =  d(z)\det(wI_m-\Phi(z)) = \sum_{j=0}^{M} q_j(z)w^j.
	\]
	By Cayley-Hamilton Theorem, $q_z (\Phi(z))=\sum_{j=0}^{M} q_j(z)\Phi(z)^j=0$ and therefore $q(z,\Phi(z))=0$ for all $z \in \mathbb{D}$. Now for $\gamma \in \mathbb{C}^M$ and $\lambda \in \mathbb{D}$, 
	$$q(M_z,M_\Phi(z))^*\gamma s_\lambda = \overline{q(\lambda,\Phi(\lambda))}\gamma s_\lambda =0.$$
	Therefore, $q(M_z,M_{\Phi})=0$. Since $\fp$ is the minimal polynomial for $(M_z, M_{\Phi})$, $\Vfp(\fp)$  is a subset of $Z(q)$.  Hence there exist an $\alpha  = (\alpha_1,...,\alpha_s) \in \N_+ ^s$ and a polynomial $r$ such that $\fp_j$ does not divide $r$ for each $j$ and 
	\begin{equation}
	\label{eq:4polys-alt}
	d(z)\det(w-\Phi(z))=q(z,w) = \fp_1^{\alpha_1} (z,w)\, \cdots \fp_s^{\alpha_s}(z,w) \, r(z,w).
	\end{equation}
	For $(\lambda,\mu)\in \overline{\mathbb D}\times\overline{\mathbb D}$,  $\mu$ is in the spectrum of $\Phi(\lambda)$ if and only if $q(\lambda,\mu)=0$. In particular, $q(z,w)$ is a polynomial whose zero set in $\mathbb D\times\mathbb C$ is the set $\{(z,w): z\in\mathbb D, \, w\in \sigma(\Phi(z))\} \subseteq \Vfp(\fp)$. Observe  $Z(r)\cap [\mathbb D\times\C] \subseteq Z(q) \cap [\mathbb D\times\C]\subseteq \Vfp(\fp)$. On the other hand, $r$ can have only finitely many zeros in $\Vfp(\fp)$ as otherwise $r$ has infinitely many zeros on some $\Vfp(\fp_j)$ and, by Lemma \ref{lem:divide} item (\ref{it:divide}) $\fp_j$ divides $r$. Hence
	$r(z,w)$ has only finitely many zeros in $\mathbb H=\mathbb D\times \C$.  We conclude there are only finitely many $z\in \mathbb D$ such that $r_z(w)=r(z,w)$ has a zero and consequently $r$ depends on $z$ only so that $r(z,w)=r(z)$.  Thus, for $\lambda\in \mathbb D$, the characteristic polynomial $f_\lambda (w)$ of $\Phi(\lambda)$
	satisfies
	\begin{equation}\label{eq:characteristic poly}
	f_\lambda(w)=\det(w-\Phi(\lambda))= c(\lambda) \fp_{1,\lambda}^{\alpha_1} (w)\, \cdots \fp_{s,\lambda}^{\alpha_s}(w),
	\end{equation}
	for a constant (in $w$) $c(\lambda)$. 
	
	
	Let  $\Lambda $ be the  set of all  $\lambda \in \mathbb{D}$ for which $\fp_\lambda$ has $\sum_{j=1}^s m_j$ distinct zeros. By Lemma \ref{lem:divide} item (\ref{it:Lambda cofinite}),  $\Lambda \subseteq \mathbb{D}$ is cofinite. For $\lambda\in \Lambda$, the polynomial $\fp_\lambda$ has distinct zeros and by \eqref{eq:pmzphi-alttwo},  $\fp_{\lambda}(\Phi(\lambda))=0$. Hence, $\Phi(\lambda)$ is diagonalizable and,  for given   $\mu_j \in Z(\fp_{j,\lambda})$, the dimension of the  eigenspace of  $\Phi(\lambda)$ at $\mu_j$ is $\alpha_j$. Thus $\Phi(\lambda)$ is similar to 
	\[
	\bigoplus_{j=1}^s \bigoplus_{\mu_j\in Z(\fp_{j,\lambda})} \mu_j I_{\alpha_j}.
	\]
	
		Let $(\lambda,\mu) \in Z(\fp_j)$ be such  that $\frac{\partial \fp}{\partial w}|_{(\lambda,\mu)} \neq 0$. The minimal polynomial for $\Phi(\lambda)$ has a zero of multiplicity 1 at $\mu$, since it  divides $\fp_\lambda$. Hence $\Phi(\lambda)$ is similar to $\mu I_{\alpha_j} \oplus J$ where  the spectrum of $J$ does not contain $\mu$. Therefore, the kernel of $\Phi(\lambda)- \mu$  has dimension $\alpha_j$. 
\end{proof}

\begin{proposition} \label{jointkernelconstant1}
	Let $V=(S,T)$ be a pure $\fp$-isopair of finite bimultiplicity   and  suppose $\fp=\fp_1\,\fp_2\, \cdots \fp_s$  a product of distinct irreducible factors. For each $j$ and $(\lambda,\mu)\in \Vfpj$ such that $\frac{\partial \fp}{\partial w}|_{(\lambda,\mu)} \neq 0$, the dimension of the intersection of $\ker(S-\lambda)^*$ and $\ker(T-\mu)^* $ is a nonzero constant. 
\end{proposition}

\begin{proof}
 By the standard model theory for pure isopairs with finite bimultiplicity,  there exists an  $M \times M$ matrix-valued  rational inner function $\Phi$ such that $V=(S,T)$ is unitarily equivalent to $(M_z, M_{\Phi})$ on $\cH^2_{\mathbb{C}^M}$ and $\fp(M_z,M_\Phi)=0$.   Let $(\lambda,\mu)\in \Vfpj$ be a regular point for $\fp$.  Observe that for any $\gamma \in \ker(\Phi(\lambda)-\mu)^*$, both  $(S-\lambda)^*s_\lambda \gamma =0$ and $(T-\mu)^*s_\lambda \gamma =0$. Hence $s_\lambda \gamma \in \ker(S-\lambda)^*\cap \ker(T-\mu)^*$.  Now suppose  $f \in \ker(S-\lambda)^* \cap \ker(T-\mu)^*$.   Since $(S-\lambda)^*f=0$, there is a vector $\gamma\in\C^N$ such that $f=s_\lambda \gamma$. Thus,  $0=(T-\mu)^* s_\lambda \gamma = s_\lambda (\Phi(\lambda)^*-\mu^*)\gamma$.  Hence
  \[s_\lambda  \ker(\Phi(\lambda)-\mu)^* =  \ker(S-\lambda)^* \cap \ker(T-\mu)^*.\]
Since $\dim\ker(\Phi(\lambda)-\mu)^* = \dim\ker(\Phi(\lambda)-\mu)$, we have 
\begin{equation} \label{eq: dimkereq}
\dim\left[\ker(S-\lambda)^* \cap \ker(T-\mu)^*\right] = \dim\ker(\Phi(\lambda)-\mu),
\end{equation}
and hence by Proposition \ref{alphaexist} item  (\ref{it: dimkernel}),   $\dim\left[\ker(S-\lambda)^* \cap \ker(T-\mu)^*\right] =  \alpha_j.$
\end{proof}

\begin{cor}\label{cor:jointkernelconstant}
	Let $V=(S,T)$ be a pure $\fp$-isopair of finite bimultiplicity   and  suppose $\fp=\fp_1\,\fp_2\, \cdots \fp_s$  a product of distinct irreducible factors. For each $j$ and $(\lambda,\mu)\in \Vfpj$ such that $\frac{\partial \fp}{\partial z}|_{(\lambda,\mu)} \neq 0$,  dimension of the intersection of $\ker(S-\lambda)^*$ and $\ker(T-\mu)^* $ is a nonzero constant. 
\end{cor}

\begin{proof}
The proof is immediate from the symmetry of $S$ and $T$ and Proposition \ref{jointkernelconstant1}.
\end{proof}

\begin{proof}[Proof of Proposition \ref{jointkernelconstant}]
Let $(\lambda,\mu) \in \Vfpj$. If $\frac{\partial \fp}{\partial w}|_{(\lambda,\mu)} \neq 0$, then by Proposition \ref{jointkernelconstant1}, there exists a non zero constant  $\alpha_j \in \mathbb{N}^+$ such that 
\[ 
\dim(\ker(S-\lambda)^* \cap \ker(T-\mu)^*) = \alpha_j.
\]
If $\frac{\partial \fp}{\partial z}|_{(\lambda,\mu)} \neq 0$, then by Corollary \ref{cor:jointkernelconstant}, there exists a non zero constant $\beta_j \in \mathbb{N}^+$ such that 
\[ 
\dim(\ker(S-\lambda)^* \cap \ker(T-\mu)^*) = \beta_j.
\]
Note that, since  $\fp$ is square free, so is $\fp_j$ and hence there are infinitely many points  in  $ \Vfpj$ such that  both partial derivatives $\frac{\partial \fp}{\partial z} |_{(z_0,w_0)}$ and $\frac{\partial \fp}{\partial w} |_{(z_0,w_0)}$ do not vanish.   If $(\lambda,\mu)$ is a regular point for $\fp$ such that $\frac{\partial \fp}{\partial z} |_{(\lambda,\mu)} \neq 0$ and $\frac{\partial \fp}{\partial w} |_{(\lambda,\mu)} \neq 0$, then $\alpha_j=\beta_j$.  Therefore, if $(\lambda,\mu) \in \Vfpj$ is a regular point for $\fp$, then dimension of the intersection of $\ker(S-\lambda)^*$ and $\ker(T-\mu)^* $ is a nonzero constant. 
\end{proof}

\begin{Cor} \label{prop:defrank}
	If  $(S,T)$ is a pure $\fp$-isopair of finite bimultiplicity $(M,N)$ with rank $\alpha = (\alpha_1,...,\alpha_s)\in \N _+^s$, then 
	\begin{equation}	\label{it:rank}
	\displaystyle{ M =\sum_{j=1}^s m_j \alpha_j}    \textrm{ and } \displaystyle{ N=\sum_{j=1}^s n_j\alpha_j}  .
	\end{equation}
\end{Cor}

\begin{proof}
    First, view $(S,T)$ as $(M_z,M_{\Phi})$ where $\Phi(z)$ is an $M \times M$ matrix-valued rational inner function. By Proposition \ref{alphaexist}  item (\ref{characteristic poly}), for $\lambda\in \mathbb D$,  
    \[
    \det(w-\Phi(\lambda))= c(\lambda) \fp_{1,\lambda}^{\alpha_1} (w)\, \cdots \fp_{s,\lambda}^{\alpha_s}(w)
    \]
    for a  constant (in $w$)  $c(\lambda)$.
    Comparing the degree in $w$ on the left and the right, for all but finitely many $\lambda$, we have \[M=\sum_{j=1}^s \alpha_j m_j.\] 
	
	To see the  relation on $N$, view $\fp$ as $\fp(w,z)$ a polynomial of bidegree $(m,n)$. Note that each factor $\fp_j = \fp_j (w,z)$ has bidegree $(m_j,n_j)$. Moreover      $\fp(T,S) =0$ and $(T,S)$ has bimultiplicity $(N,M)$.  Model  $(T,S)$ as $(M_w, M_{\Psi(w)})$, where $\Psi(w)$	is an $N \times N$ matrix valued ration inner function. By Proposition \ref{alphaexist}, item (\ref{characteristic poly}), there exists  $(\beta_1,\beta_2,...,\beta_s) \in \N_+ ^s$ such that for   $\mu \in \mathbb{D}$, 
    \begin{equation} \label{eq:reversecharacteristicpoly}
    \det(z-\Psi(\mu))= c'(\mu) \fp_{1,\mu}^{\beta_1} (z)\, \cdots \fp_{s,\mu}^{\beta_s}(z)
    \end{equation}
	for a  constant (in $z$)  $c'(\mu)$. By Proposition \ref{alphaexist} item (\ref{it: dimkernel}), for $(\mu, \lambda) \in Z(\fp_j)$ that is a regular point for $\fp$, 
	\[\dim\ker(\Psi(\mu) -\lambda)= \beta_j.\] 
	Now by equation \eqref{eq: dimkereq}, 
	\[ \dim\left[\ker(S-\lambda)^* \cap \ker(T-\mu)^*\right] =  \beta_j.
	\]
	Since $(S,T)$ has rank $\alpha$, we get $\beta_j = \alpha_j$ for $j=1,...,s$ and by comparing the degree in $z$ on the left and the right of \eqref{eq:reversecharacteristicpoly}, for all but finitely many $\mu$, we have \[N=\sum_{j=1}^s \alpha_j n_j. \qedhere \]  
	\end{proof}

\begin{proposition}
 \label{prop:kgealpha+}
  If $V=(S,T)$ is a finite bimultiplicity  $k$-cyclic pure $\fp$-isopair acting on the Hilbert space $\cK$,  then for each $(\lambda,\mu) \in \Vfp(\fp)$, 
\[ 
  \dim\left(\ker(S-\lambda)^* \cap \ker(T-\mu)^*\right) \le k.
\]
 In particular,  if $\fp$ is the minimal polynomial for $V$ and if $V$ has rank $\alpha$, then $k\ge \max\{\alpha_1,...,\alpha_s\}$. 
\end{proposition}

\begin{proof}
	 Let $\{f_1,...,f_k\}$ be a cyclic set for $(S,T)$. For any $q(z,w)\in \mathbb{C}[z,w]$, $f \in \ker (S-\lambda)^* \cap \ker (T-\mu)^*$ and  $1 \leq j \leq k$,
\[ 
\begin{split}
 \langle q(S,T) f_j , f \rangle 
 = & \langle  f_j , q(S,T)^* f \rangle\\
 = & \langle  f_j , q(\lambda,\mu)^* f \rangle \\
 = & q(\lambda,\mu) \langle  f_j , f \rangle.
\end{split}
\]
If $ \dim\left(\ker (S-\lambda)^* \cap \ker (T-\mu)^*\right) > k$, then there exists a non zero vector $f \in \ker (S-\lambda)^* \cap \ker (T-\mu)^*$ perpendicular to $f_j$ for all $j$. Thus $\langle q(S,T) f_j , f \rangle = 0 $ for all $j$ and for any $q$, and hence $\langle g , f \rangle = 0 $ for any $g \in \{\sum_{j=1}^k q_j(S,T)f_j : q_j \in \C[z,w] \}$, a contradiction. Therefore, $\dim \left(\ker (S-\lambda)^* \cap \ker (T-\mu)^* \right) \leq k.$
The last statement of the proposition follows from the definition of the rank.
\end{proof}

\begin{proposition}
 \label{prop:rank is stable}
\label{prop:keeptherank}
  Suppose $V=(S,T)$ is a finite bimultiplicity pure $\fp$-isopair with minimal polynomial $\fp$ and with rank $\alpha = (\alpha_1,...,\alpha_s) \in \N_+^s$ acting on a Hilbert space $\cK$. If $\cH$ is a finite codimension $V$-invariant subspace of $\cK$, then $W=V|_{\cH}$ has rank $\alpha$ too. 
\end{proposition}

\begin{proof}
Write $W= V|_{\cH} = (S_0,T_0)$. Let $\cF = \cK \ominus \cH$. Thus $\cF$ has finite dimension and $\cK = \cH \oplus \cF$. With respect to this decomposition, write
$$ S^* =\begin{pmatrix} S_0^* & 0 \\ X^* & A^* \end{pmatrix}, \ \  T^*=\begin{pmatrix} T_0^* & 0 \\ Y^* & B^* \end{pmatrix}.$$
Observe that $\sigma(A) \times \sigma(B)$ is a finite set since $A$ and $B$ act on a finite dimensional space. Fix $1\leq j\leq s$.   Let $\Gamma$ be the set of all $(\lambda,\mu) \in \Vfp(\fp_j)$ such that the dimension of $\ker(S-\lambda)^*  \cap \ker(T-\mu)^*$ is $\alpha_j$ and  $ (\lambda,\mu) \notin \sigma(A) \times \sigma(B)$. Hence by Proposition \ref{jointkernelconstant},  $\Gamma$ contains the cofinite  set of all regular points. Since also  the set $\sigma(A) \times \sigma(B)$ is finite, $\Gamma$ is a cofinite subset of $\Vfp(\fp_j)$. Fix $ (\lambda,\mu) \in \Gamma$ and let
$$\cL = \ker(S-\lambda)^* \cap \ker(T-\mu)^*\  \textrm{ and }\  \cL_0 =\ker(S_0-\lambda)^*\cap \ker(T_0-\mu)^*.$$
Let $\cP \subseteq \cH$ be the projection of $\cL$ onto $\cH$.  Given $f\in \cL$, write $f=f_1 \oplus f_2$, where $f_1 \in \cH$ and $f_2 \in \cF$. Since $f \in \cL$, the kernel of $(S_0-\lambda)^*$ contains $f_1$. Likewise the kernel of $(T_0-\lambda)^*$ contains $f_1$. Therefore, $\cP \subseteq \cL_0$.
If $\dim(\cL_0) < \alpha_j$, then, since $\dim(\cL) = \alpha_j$, there exists a non zero vector of the form $0 \oplus v$ in $\cL$ and hence $\ker(A-\lambda)^* \cap \ker(B-\mu)^*$ is non-empty. But, $\ker(A-\lambda)^* \cap \ker(B-\mu)^*$ is empty by the choice of $(\lambda, \mu)$. Thus $\dim(\cL_0) = \alpha_j$ for almost all $(\lambda,\mu)$  in $\Vfp(\fp_j)$. Therefore  $W$ also has rank $\alpha$.
\end{proof}

\begin{cor}
 \label{cor:nearlybound}
   Suppose $V=(S,T)$ is a finite bimultiplicity pure $\fp$-isopair with minimal polynomial $\fp$ and with rank $\alpha = (\alpha_1,...,\alpha_s) \in \N_+^s$ acting on a Hilbert space $\cK$. If $\cH$ is a finite codimension $V$-invariant subspace of $\cK$, then $W=V|_{\cH}$ is at least $\beta=\max\{\alpha_1,\dots,\alpha_s\}$-cyclic.  Hence $V$ is at least nearly $\beta$-cyclic.
\end{cor}

\begin{proof}
 By Proposition \ref{prop:keeptherank}, $W$ has rank $\alpha.$ By Proposition \ref{prop:kgealpha+}, $W$ is at least $\beta$-cyclic. Thus, each restriction of $V$ to a finite codimension invariant subspace  is at least $\beta$-cyclic and hence $V$ is at least nearly $\beta$-cyclic. 
\end{proof}

\section{THE CASE $\textbf{\em p}$  IS IRREDUCIBLE}

In this section $\fp$ is an irreducible square free inner toral polynomial of bidegree $(n,m).$  

A \df{rank $\alpha$-admissible kernel}  $\K$ over $\Vfp(\fp)$
consists of an $\alpha\times m\alpha$ matrix polynomial $Q$
and an  $\alpha \times n\alpha$ matrix polynomial $P$  such that 
\[
  \frac{Q(z,w) Q(\zeta,\eta)^*}{1-z\zeta^*}
   = \K((z,w),(\zeta,\eta)) = \frac{P(z,w)P(\zeta,\eta)^*}{1-w\eta^*},\ (z,w),(\zeta,\eta) \in \Vfp(\fp)
\]
 where $Q$ and   $P$ have full rank $\alpha$ at some point in $\Vfp(\fp)$. In particular, at some point $x\in \Vfp(\fp)$ the matrix
$\K(x,x)$ has full rank $\alpha$ \cite{JKM}.  An $\alpha\times \alpha$ matrix-valued kernel on a set $\Omega$ has {\it full rank} at $x \in \Omega$, if $\K(x,x)$ has full rank $\alpha$.\index{full rank} \index{full rank kernel}We refer to $(\K,P,Q)$ as an \df{$\alpha$-admissible triple}. \index{admissible triple}

Let $\cH^2(\K)$ denote the Hilbert space associated to the  rank $\alpha$ admissible kernel $\K$. For a point $y\in \Vfp(\fp)$, denote by $\K_y$ the $\alpha \times \alpha$ matrix
  function on $\Vfp(\fp)$ defined
 by $\K_y(x) =\K(x,y)$. Elements of $\cH^2(\K)$ are $\C^\alpha$ vector-valued functions on $\Vfp(\fp)$
 and the linear span of $\{\K_y \gamma : y\in \Vfp(\fp), \, \gamma\in\C^\alpha\}$
 is dense in $\cH^2(\K)$. Note that the operators $X$ and $Y$ determined densely on $\cH^2(\K)$ by
 $X \K_{(\lambda,\mu)}\gamma =\lambda^* \K_{(\lambda,\mu)}\gamma$ and 
 $Y \K_{(\lambda,\mu)}\gamma =\mu^* \K_{(\lambda,\mu)}\gamma$  are contractions.
 By   Theorem \eqref{thm:KtoST} item  \eqref{it:Xbdd} below, $X^*$ is a bounded operator on $\cH^2(\K)$. Further for $f\in \cH^2(\K)$,
\[
 \langle X^* f, \K_{\lambda,\mu} \gamma \rangle = \lambda \langle f(\lambda,\mu),\gamma\rangle.
\]
 Hence $X^*$ is the operator of multiplication by $z$ on $\cH^2(\K)$. Likewise, $Y^*$ is  a bounded operator on  $\cH^2(\K)$ and it is the multiplication by $w$ on $\cH^2(\K)$.

\begin{theorem}
\label{thm:KtoST}
 If $\K$ is a rank $\alpha$-admissible kernel over $\Vfp(\fp)$, then 
\begin{enumerate}[(i)]
 \item \label{it:Xbdd} $X$ is bounded on the linear span of $\{\K_y \gamma : y\in \Vfp(\fp), \, \gamma\in\C^\alpha\}$;
 \item 
 \label{it:znqej}
     for each $1\le j \le m\alpha$ and each positive integer $n$,  the vector $z^n Qe_j$  ($Qe_j$ is the $j$-th column of $Q$)
 lies in $\cH^2(\K)$; 
 \item \label{it:szegoQ*}
 the span of $\displaystyle{\left\{s_\lambda Q(\lambda,\mu)^* \gamma : (\lambda,\mu) \in \Vfp(\fp), \gamma \in \mathbb{C}^\alpha\right\}}$ is dense in $\cH^2_{\mathbb{C}^{m\alpha}}$;
 \item
  \label{it:onB}
   the set $\mathscr B=\{z^n Qe_j: n\in\mathbb N,\, 1\le j\le  m\alpha\}$  is an orthonormal basis for $\cH^2(\K)$; and
   \item
  \label{it:STpureiso} operators $S$ and $T$ densely defined on $\mathscr  B$ by $Sf=zf$ and $Tf=wf$ extend to a pair of  pure isometries on $\cH^2(\K)$.
  \end{enumerate} 
\end{theorem}

\begin{proof}
 For a finite set of points $(\lambda_1,\mu_1),...,(\lambda_n,\mu_n)  \in \Vfp(\fp)$, and $\gamma_1,...,\gamma_n \in \C^{\alpha}$, observe that
\[
\begin{split}
 \langle (I-X^*X)\sum_{j=1}^n \K_{(\lambda_j,\mu_j)} \gamma_j , \sum_{k=1}^n \K_{(\lambda_k,\mu_k)} \gamma_k \rangle = & \sum_{j,k=1}^n \langle(1-\lambda_k \overline{\lambda_j}) \K_{(\lambda_j,\mu_j)}(\lambda_k,\mu_k) \gamma_j, \gamma_k \rangle\\
        = & \sum_{j,k=1}^n \langle Q(\lambda_k,\mu_k)Q^*(\lambda_j,\mu_j) \gamma_j, \gamma_k \rangle\\
        = & \  \langle \sum_{j=1}^n Q^*(\lambda_j,\mu_j) \gamma_j , \sum_{k=1}^n Q^*(\lambda_k,\mu_k) \gamma_k \rangle \\
        \geq & \ 0.
\end{split}
\]
Therefore, $X$ is bounded on the linear span of $\{\K_y \gamma : y\in \Vfp(\fp), \, \gamma\in\C^\alpha\}$.

 To prove item \eqref{it:znqej}, note that by  \cite[Theorem 4.15]{VIP}, if $f$ is a $\mathbb{C}^{\alpha}$ valued function defined on $\Vfp(\fp)$ and if $\K((z,w),(\zeta,\eta)) - f(z,w)f(\zeta,\eta)^*$ is a (positive semidefinite) kernel function  then $f \in \cH^2(\K)$. Since
 \[
\begin{split}
  \K((z,w),(\zeta,\eta))- (z\zeta^*)^n Q(z,w)Q^*(\zeta,\eta)  = & \sum_{j=1}^{n-1} (z\zeta^*)^j Q(z,w)Q^*(\zeta,\eta) \\   + &
  (z\zeta^*)^{n+1}\K((z,w),(\zeta,\eta))
\end{split}
\]
 is positive semidefinite, it follows that  $z^n Qe_j \in \cH^2(\K).$  
 
 By  a result in  \cite[Lemma 4.1]{JKM}, there exists a cofinite subset $\Lambda \subset \mathbb D$ 
 such that for each  $\lambda\in \Lambda$ there
 exist distinct points $\mu_1,\dots,\mu_m \in\mathbb D$ such that 
 $(\lambda,\mu_j)\in \Vfp(\fp)$ and the $m\alpha\times m\alpha$ matrix,
\[
  R(\lambda):=\begin{pmatrix} Q(\lambda,\mu_1)^* & \dots & Q(\lambda,\mu_m)^*\end{pmatrix}
\]
 has full rank.  Define a map $U$ from $\cH^2(\K)$ to $\cH^2_{\C^{m\alpha}}$ by
\[
 U \K_{(\lambda,\mu)} (z,w) \gamma = s_\lambda(z) Q(\lambda,\mu)^* \gamma.
\]
Observe that for  $(\lambda_1,\mu_1), (\lambda_2,\mu_2) \in \mathbb{D}^2$  and $\gamma, \delta \in \mathbb{C}^\alpha$,
\[
\begin{split}
\langle U\K_{(\lambda_1,\mu_1)} (z,w) \gamma , U\K_{(\lambda_2,\mu_2)} (z,w)\delta \rangle 
= & \langle s_{\lambda_1} (z) Q(\lambda_2,\mu_2)Q^*(\lambda_1,\mu_1) \gamma , s_{\lambda_2} (z)\delta \rangle \\
 = & \delta^* Q(\lambda_2,\mu_2)Q^*(\lambda_1,\mu_1) \gamma \langle s_{\lambda_1} (z) , s_{\lambda_2} (z) \rangle \\
 = & \dfrac{\delta^* Q(\lambda_2,\mu_2)Q^*(\lambda_1,\mu_1) \gamma}{1-\overline{\lambda_1}\lambda_2}\\
 = & \delta^* \K((\lambda_2,\mu_2),(\lambda_1,\mu_1)) \gamma\\
 = & \langle \K_{(\lambda_1,\mu_1)} (z,w) \gamma , \K_{(\lambda_2,\mu_2)} (z,w) \delta \rangle.
\end{split}
\]
Therefore, $U$ is an isometry and hence a unitary, onto its range. Given $\lambda \in \mathbb{D}$, the span of
\[
 \{U \K_{(\lambda,\mu_j)}\gamma: \mu_j\in Z(\fp_\lambda), \, \gamma \in\C^\alpha\}
\]
 is equal to  $s_\lambda$ times the span of 
\[
 \{ Q(\lambda,\mu_j)^* e_k : 1 \leq j \leq m , 1 \leq k \leq \alpha \} \subseteq \mathbb{C}^{m\alpha}.
\]
If $\lambda \in \Lambda$, then  $R(\lambda)$ has full rank. Thus for such $\lambda$, the span of  $\{ Q(\lambda,\mu)^* \gamma : \mu \mbox{ such that } (\lambda,\mu) \in \Gamma, \gamma \in \mathbb{C}^{\alpha}  \}$ is all of $\mathbb{C}^{m\alpha}$.  Since $\Lambda \subseteq \mathbb{D} $ is cofinite, $\{s_\lambda \C^{m\alpha} : \lambda \in \Lambda\}$ is dense in $\cH^2_{\mathbb{C}^{m\alpha}}$. Since, 
\[ 
\{s_\lambda \C^{m\alpha} : \lambda \in \Lambda\} \subseteq \mbox{span} \{s_\lambda Q(\lambda,\mu)^*\gamma : (\lambda,\mu) \in \Vfp(\fp), \gamma \in \mathbb{C}^\alpha\}, 
\]
 the span of $\{s_\lambda Q(\lambda,\mu)^*\gamma : (\lambda,\mu) \in \Vfp(\fp), \gamma \in \mathbb{C}^\alpha\}$ is also dense in $\cH^2_{\mathbb{C}^{m\alpha}}$, proving item \eqref{it:szegoQ*}.  Moreover, it proves that  $U$ is onto and hence  unitary. 
 
 Let $q_k$ denote the $k$-th column of $Q$. Thus $q_k=Qe_k$.  Note that, for any $a \in \mathbb{N}$ and $1\leq j \leq m\alpha$, 
\[
\begin{split}
    \langle U^*z^ae_j(\zeta,\eta) , e_k \rangle
    = & \langle U^*z^ae_j , \K_{(\zeta,\eta)}e_k \rangle \\
    = & \langle z^ae_j , U \K_{(\zeta,\eta)}e_k \rangle \\
    = & \displaystyle{\sum_{i=1}^{m\alpha} \langle z^a e_j , ( s_{\zeta}q_i^*(\zeta,\eta) e_k) e_i \rangle }\\
    = & \langle q_j(\zeta,\eta) {\zeta}^a, e_k \rangle\\
    = & \langle (z^a q_j)(\zeta,\eta) , e_k \rangle
\end{split}
\]
and hence it  follows that $ U^*z^ae_j = z^a q_j$ and $U z^a q_j = z^ae_j$. In particular, $\{z^a q_j : a\in\mathbb N, \, 1\le j \le m\alpha\}$
is an orthonormal basis for $\cH^2(\K)$ completing the proof of item \eqref{it:onB}.

To prove item \eqref{it:STpureiso},  observe that $M_zU =US$ on $\mathscr B$ and then extending to $\cH^2(\K)$, it is true on  $\cH^2(\K)$ too.
It is now evident that $S$ is a pure isometry of multiplicity $m\alpha$ with wandering
 subspace $\{Q\gamma: \gamma\in\C^{m\alpha}\}$ (the span of the columns of $Q$).  Likewise for $T$ by  symmetry. 
\end{proof}

 \begin{proposition}[\cite{AM2}]
 	\label{prop:finiterational} 
 	Suppose $\Phi$ is an  $M\times M$ matrix-valued  rational  inner function and the
 	pair $(M_z,M_\Phi)$ of multiplication operators on $\cH^2_{\C^M}$.  If the rank of the projection $I-M_\Phi M_\Phi^*$ is $N$,
 	then there exists a unitary matrix  $U$ of size $(M+N)\times (M+N)$,
 	\[
 	U = \begin{matrix} \begin{matrix} M & N \end{matrix}  & {}\\ 
 	\begin{pmatrix} A & B \\ C & D \end{pmatrix} &\begin{matrix} M\\ N\end{matrix}
 	\end{matrix},
 	\]
 	such that
 	\[
 	\Phi(z) = A + zB(I-zD)^{-1}C.
 	\]
 \end{proposition}

\begin{proposition}
\label{prop:rankV}
 If $V=(S,T)$ is a finite bimultiplicity $(M,N)$ pure $\fp$-isopair of rank  $\alpha$, modeled as $(M_z,M_\Phi)$ 
 on $\cH^2_{\C^M}$, where $\Phi$ is an
  $M\times M$ matrix-valued   rational  inner function, then $M=m\alpha$ and 
\begin{enumerate}[(i)]
 \item there exists an $\alpha \times m\alpha$
 matrix  polynomial $Q$ such that $Q(z,w)$ has full rank at almost all points of $\Vfp(\fp)$;
 \item
  \label{it:Q} 
    for $(z,w)\in \Vfp(\fp)$
\[
 Q(z,w)(\Phi(z)-w)=0;
\]
\item there exists an $\alpha \times n\alpha$
matrix  polynomial $P$ such that $P(z,w)$ has full rank at almost all points of $\Vfp(\fp)$ and an $\alpha$-admissible kernel $\K$ such that 
\[
 \frac{Q(z,w) Q(\zeta,\eta)^*}{1-z\zeta^*}
=   \K((z,w),(\zeta,\eta)) = \frac{P(z,w)P(\zeta,\eta)^*}{1-w\eta^*} \mbox{ on } \Vfp(\fp) \times \Vfp(\fp).
\]

\end{enumerate}
\end{proposition}

\begin{remark}
   The triple $(\K,P,Q)$ in   Proposition \ref{prop:rankV} is a rank $\alpha$-admissible triple.  
\end{remark}

\begin{proof}
 Applying Corollary \ref{prop:defrank} to irreducible $\fp$ gives $M=m\alpha$.  Let   $\Lambda$ denote the  set of  $\lambda \in \mathbb{D}$ such that  $\fp_\lambda$ has $m$ distinct zeros. By Lemma \eqref{lem:divide} item \eqref{it:Lambda cofinite} $\Lambda$ is cofinite.   Let 
\[
 \Gamma =\{(\lambda,\mu): \lambda\in \Lambda, \, \mu \in Z(\fp_\lambda)\}.
\]
By Proposition \ref{alphaexist} item \eqref{it:diagonalizable}, for each $(\lambda,\mu) \in \Gamma$, the matrix  $\Phi(\lambda)$ is diagonalizable and $\Phi(\lambda)-\mu$ has an $\alpha$ dimensional kernel.
Now fix $(\lambda_0,\mu_0)\in \Gamma$. 
  Hence there exist  unitary matrices $\Pi$  and $\Pi_*$  such that
\[
 \Pi_*  (\Phi(\lambda_0)-\mu_0)\Pi 
 = \begin{pmatrix} 0_\alpha & 0 \\ 0 & A \end{pmatrix},
\] 
where $A$ is $(m-1)\alpha\times (m-1)\alpha$ and invertible.  Let 
\[
 \Sigma(z,w) = \Pi_*(\Phi(z)-w)\Pi. 
\]
 For $(\lambda,\mu)\in \Gamma,$ the matrix $\Sigma(z,w)$ has an $\alpha$ dimensional kernel.
Write,
\[
\Sigma(z,w)
 = \begin{pmatrix} E(z)-w & G(z)\\ H(z) & L(z)-w \end{pmatrix},
\]
 where $E$ is $\alpha\times \alpha$ and  $L$ is of size $(m-1)\alpha \times (m-1)\alpha$.
 By construction $L(z)-w$ is invertible at $(\lambda_0,\mu_0)$ and the other entries
 are $0$ there.  In particular, $L(\lambda)-\mu$ is invertible for almost all points
 $(\lambda,\mu)\in \Vfp(\fp)$. 
 Moreover, if $L(z)-w$ is invertible, then  
\[
 {\tiny  \Sigma(z,w) =  \begin{pmatrix} I&  G(z)  \\ 0 & L(z)-w \end{pmatrix}
 \begin{pmatrix} \Psi(z,w)  & 0 \\ 0 & I \end{pmatrix}
    \begin{pmatrix} I & 0 \\ (L(z)-w)^{-1} H(z) & I \end{pmatrix},}
\]
where 
\[
\Psi(z,w)=E(z)-w - G(z)(L(z)-w)^{-1} H(z).
\]
 Thus, on the cofinite subset of $\Vfp(\fp)$ where $L(\lambda)-\mu$ is invertible
 and $\Sigma(\lambda,\mu)$ has an $\alpha$ dimensional kernel, $\Psi(\lambda,\mu)=0$
 and moreover,
\[
  \begin{pmatrix} I_\alpha & - G(\lambda)(L(\lambda)-\mu)^{-1}  \end{pmatrix} 
     \Pi_*  (\Phi(\lambda)-\mu)=0.
\]
 Let 
\[
 \cQ(z,w) = \begin{pmatrix} I_\alpha & - G(z)(L(z)-w)^{-1}    \end{pmatrix} \Pi_*.
\]
 It follows that
\[
\cQ(z,w) (\Phi(z) -w) = 0
\]
for almost all points in $\Vfp(\fp)$. After multiplying  $\cQ$ by an appropriate scalar polynomial
 we obtain an $\alpha \times m\alpha$ matrix polynomial $Q(z,w)$ that has
 full rank at almost all points of $\Vfp(\fp)$ and satisfies
\[
 Q(z,w) (\Phi(z)-w) =0 
\] for all $(z,w)\in \Vfp(\fp)$. 
 
 Since $T$ has multiplicity $N$, the operator $M_\Phi$ also has multiplicity $N$  and hence the projection $I -M_\Phi M_\Phi^*$ has rank $N$. By Theorem \ref{prop:finiterational}, there exists a unitary matrix $U$ of size $(M+N)\times (M+N)$, 
 \[
 U = \begin{matrix} \begin{matrix} M & N \end{matrix}  & {}\\ 
 \begin{pmatrix} A & B \\ C & D \end{pmatrix} &\begin{matrix} M\\ N\end{matrix}
 \end{matrix},
 \]
 such that
 \[
 \Phi(z) = A + zB(I-zD)^{-1}C.
 \]
 Define $P$ by $P(z,w) = Q(z,w)B(I-zD)^{-1}$ and verify, for $(z,w)\in \Vfp(\fp)$, 
 \[
 \begin{pmatrix} Q & zP \end{pmatrix} \begin{pmatrix} A & B \\ C & D \end{pmatrix}
 = \begin{pmatrix} wQ & P \end{pmatrix}
 \textrm{ on } \Vfp(\fp).  \]
 It follows that, for $(\zeta,\eta) \in \Vfp(\fp)$, 
 \[
 Q(z,w)Q(\zeta,\eta)^* + z \zeta^* P(z,w)P(\zeta,\eta)^* = w\eta^* Q(z,w)Q(\zeta,\eta)^* + P(z,w)P(\zeta,\eta)^*.
 \]
 Rearranging gives,
 \[
 \frac{Q(z,w) Q(\zeta,\eta)^*}{1-z\zeta^*}  = \K((z,w),(\zeta,\eta)
 = \frac{P(z,w)P(\zeta,\eta)^*}{1-w\eta^*} \mbox{ on } \Vfp(\fp) \times \Vfp(\fp). 
 \]
 Finally, if $(\zeta,\eta)\in\Vfp(\fp)$ is such that $Q(\zeta,\eta)$ has full rank $\alpha$, then $P(\zeta,\eta)P(\zeta,\eta)^*$ also has full rank $\alpha$. Therefore, $P(\zeta,\eta)$ also has full rank $\alpha$ and hence $\K$ is a rank $\alpha$-admissable kernel.
 \end{proof}

\begin{theorem}
\label{thm:STtoK}
 If $V=(S,T)$ is a finite bimultiplicity $(M,N)$  pure $\fp$-isopair with  rank $\alpha$, then there exists a rank $\alpha$-admissible triple $(\K,P,Q)$
 such that $V$ is unitarily equivalent to the operators of multiplication by
 $z$ and $w$ on $\cH^2(\K)$.  
\end{theorem}

\begin{proof}
 Note that $(S,T)$ is unitarily equivalent to $(M_z,M_\Phi)$ on $\cH^2_{\mathbb{C}^M}$, where $\Phi$ is an $M\times M$ matrix-valued rational inner function.  By Proposition \ref{prop:rankV}, there exists a rank $\alpha$-admissible triple $(\K,P,Q)$ such that   
 \begin{equation}\label{eq:before*}
 Q(z,w) (\Phi(z)-w) =0 
 \end{equation} 
 for all $(z,w)\in \Vfp(\fp)$. Define
 \[
 U:\cH^2_{\C^M} \to \cH^2(\K)
 \]
 on the span of 
 \[
 \mathcal B= \{ s_\zeta Q^*(\zeta,\eta)\gamma: (\zeta, \eta) \in \Vfp(\fp),  \  \gamma\in\C^\alpha\} \subseteq \cH^2_{\C^M}
 \]
 by
 \[
 U s_\zeta (z)Q^*(\zeta,\eta)\gamma  = \K_{(\zeta,\eta)} (z,w)\gamma.
 \]
 For $(\zeta, \eta) \in \Vfp(\fp)$ and $ \gamma_j \in \C^\alpha$ for $1 \leq j\leq 2$,  
 \[
 \begin{split}
 \langle U s_{\zeta_1}(z) Q^*(\zeta_1,\eta_1)\gamma_1, U s_{\zeta_2}(z) Q^*(\zeta_2,\eta_2)\gamma_2 \rangle 
 = & \langle \K_{(\zeta_1,\eta_1)} (z,w)\gamma_1 , \K_{(\zeta_2,\eta_2)} (z,w)\gamma_2 \rangle\\
 = & \langle \K_{(\zeta_1,\eta_1)} (\zeta_2,\eta_2)\gamma_1 , \gamma_2 \rangle\\
 = & \langle s_{\zeta_1} (\zeta_2)Q(\zeta_2,\eta_2)Q^*(\zeta_1,\eta_1)\gamma_1 , \gamma_2 \rangle \\
 = & \langle  s_{\zeta_1} (z)Q^*(\zeta_1,\eta_1)\gamma_1,  s_{\zeta_2}(z) Q^*(\zeta_2,\eta_2)\gamma_2 \rangle.
 \end{split}
 \] 
 Hence $U$ is an isometry. By Theorem \ref{thm:KtoST} item \eqref{it:szegoQ*} the span of $\mathcal B$ is dense in $\cH^2_{\mathbb{C}^M}$. Moreover, the
 range of $U$ is  dense in $\cH^2(\K)$. Thus,  $U$ is a unitary. Rewrite  \eqref{eq:before*} as,
\begin{equation}\label{eq:*}
w^*Q^*(z,w) = \Phi^*(z)Q^*(z,w).
\end{equation}

Let $\tilde{M}_z$ and $\tilde{M}_w$ be the operators of multipliction by $z$ and $w$ on $\cH^2(\K)$ respectively.  For $(\zeta , \eta) \in \Vfp(\fp)$  and $ \gamma \in \mathbb{C}^{\alpha}$, using \eqref{eq:*},  observe that, 
 \[
 \begin{split}
 \tilde{M}_w^*U(s_\zeta (z)Q^*(\zeta,\eta)\gamma) =  &\  \tilde{M}_w^*(\K_{(\zeta,\eta)} (z,w)\gamma) \\
 = &\  \bar{\eta}\K_{(\zeta,\eta)} (z,w)\gamma \\
 = & \  \bar{\eta}U(s_\zeta Q^*(\zeta,\eta)\gamma)\\
 = & \ U(s_\zeta (z)\bar{\eta}Q^*(\zeta,\eta)\gamma)\\
 = & \ U (s_\zeta (z){\Phi(\zeta)}^* Q^*(\zeta,\eta)\gamma)\\
 = & \ U M_{\Phi}^*(s_\zeta(z) Q^*(\zeta,\eta)\gamma).
 \end{split}
 \]
 Similarly, 
 \[  \tilde{M}_z^*U(s_\zeta (z)Q^*(\zeta,\eta)\gamma) =  UM_z^*(s_\zeta(z) Q^*(\zeta,\eta)\gamma).
 \]
 Therefore,
 $UM_z^* =\tilde{M}_z^*U$ and $UM_\Phi^*=\tilde{M}_w^*U$ on the span of $\mathcal B$, and hence on  $\cH^2_{\C^M}$.  Thus our original $(S,T)$ is unitarily equivalent to $(\tilde{M}_w,\tilde{M}_w)$  on $\cH^2(\K)$. 
 \end{proof}

\begin{definition}\rm
If $\cB$ is a subspace of vector space $\cX$, then the \df{codimension} of $\cB$ in $\cX$ is the dimension of the quotient space $\cX/\cB$.
\end{definition}

\begin{lemma}
 \label{lem:cofinite}
   Suppose $\cX$ is a vector space (over $\C$) and $\cQ $ and $\cB$ are subspaces of $\cX$. If $\cQ \subset \cB$ and $\cQ$ has finite codimension in $\cX$, then $\cQ$ has finite codimension in $\cB$.
\end{lemma}


\begin{lemma}
 \label{lem:cofinite+}
 Suppose $\cK$ is a Hilbert space and $\cQ \subset \cB \subset \cK$ are linear subspaces (thus not necessarily closed) and let $\overline{\cQ}$ denote the closure of $\cQ$. If $\cQ$ has finite codimension in $\cB$ and if $\cB$ is dense in $\cK$, then there exists a finite dimensional subspace $\cD$ of  $\cK$ such that $\cK = \overline{\cQ} \oplus \cD$.
\end{lemma}


\begin{theorem}\label{S,T nearly alpha cyclic}
If  $\K$ is a rank $\alpha$ admissible kernel function  defined on $\Vfp(\fp)$  and  $S=M_z$, $T=M_w$ are  the operators of multiplication by $z$ and $w$ respectively on $\cH^2(\K)$,  then the pair $(S,T)$  is  nearly $\alpha$-cyclic.
\end{theorem}

\begin{proof} 

Since $\K$ is a rank $\alpha$ admissible kernel, there exist matrix polynomials $Q$ and $P$ of size $\alpha \times m\alpha$ and $\alpha \times n\alpha$ respectively, such that  
$$\K((z,w),(\zeta,\eta))= \dfrac{Q(z,w)Q^*(\zeta,\eta)}{1-z\bar{\zeta}}=\dfrac{P(z,w)P^*(\zeta,\eta)}{1-w\bar{\eta}}, \ (z,w), (\zeta,\eta) \in \Vfp(\fp) $$ and $Q$ and $P$ have full rank $\alpha$ at some point in $\Vfp(\fp)$. Fix $(\zeta,\eta) \in \Vfp(\fp)$ so that $Q(\zeta,\eta)$ has full rank $\alpha$.  By  the definition of $\K$ and \cite[Lemma 3.3]{JKM},   $\K((z,w) ,(\zeta,\eta))$ has full rank $\alpha$ at almost all points in $\Vfp(\fp)$. Let  
\[Q_0=Q_0 (z,w) = Q(z,w)Q^*(\zeta,\eta).
\]
Then $Q_0 e_j = (1-S\overline{\zeta})\K_{(\zeta,\eta)} e_j$.  By Theorem \eqref{thm:KtoST} item \eqref{it:znqej},  $Q_0 e_j$, the  $j^{th}$ column of $Q_0$, is also in $\cH^2(\K)$.  Letting  $\tilde{q} =\tilde{q}(z,w)$ to be the determinant of $Q_0$, since $\K((z,w) ,(\zeta,\eta))$ has full rank $\alpha$ at almost all points in $\Vfp(\fp)$, $\tilde{q}$ is nonzero except for finitely many points in $\Vfp(\fp)$. Thus, $\fp$ and $\tilde{q}$ have only finitely many common zeros in $\Vfp(\fp)$. By Lemma \ref{lem:divide} item \eqref{it:intersectD}, $\fp$ and $q$ are relatively prime.  Let $I$ be the ideal generated by $\fp$ and $\tilde{q}$. By Proposition \ref{prop:finite-codim} item \eqref{it:finite},  $\mathbb{C}[z,w]/I$ is finite dimensional. Observe that 
\[ \displaystyle{ \tilde{q}_j = \tilde{q} e_j = Q_0\textrm{Adj} (Q_0) e_j = \sum_{k=1}^{\alpha} b_{kj}Q_0e_k \in \cH^2(\K)},
\]
 where $b_{kj}$ is the $(k,j)$-entry of $\textrm{Adj} (Q_0)$. If $\vec{r}$ is an $\alpha\times 1$ matrix polynomial with entries $r_j$, then
\begin{equation}\label{eq:tildeq} \displaystyle{ \vec{r}\tilde{q} = \sum_{j=1}^\alpha r_j \ \textrm{Adj} (Q_0)Q_0 e_j \in \cH^2(\K)}.\end{equation} Since $\mathbb{C}[z,w]/I$ is  finite dimensional, there is a finite dimensional subspace $\mathscr S \subseteq \mathbb{C}[z,w]$ such that 
\[
\{r\tilde{q}+s\fp+t \ |\  r,s\in \mathbb{C}[z,w] , t \in \mathscr S \} = \mathbb{C}[z,w].
\]
Therefore
\[ 
\left\{\vec{r}\tilde{q}+\vec{s}\fp+\vec{t} \ : \  \vec{r},\vec{s} \textrm{ are vector polynomials }, \vec{t} \in \bigoplus_1^\alpha \mathscr S \right\} = \bigoplus_1^\alpha \mathbb{C}[z,w].
\]
and hence  the span $\cQ$  of $ \{r_1\tilde{q}_1,..., r_\alpha \tilde{q}_\alpha\ : r_1,..,r_\alpha \in \mathbb{C}[z,w] \}$ is of finite codimension in ${\bigoplus_1^\alpha  \mathbb{C}[z,w]}$. 

Let $\mathcal B = \vee\{z^n Qe_j: n\in\mathbb {N},\, 1\le j\le  m\alpha\} \subseteq \bigoplus_1^\infty \mathbb{C}[z,w]$.  By equation (\ref{eq:tildeq}) $\cQ \subset \mathcal B$. By Lemma \ref{lem:cofinite}, $\cQ$ has finite codimension in ${\bigoplus_1^\alpha  \mathbb{C}[z,w]}$. Moreover, $\cB$ is dense in $\cH^2(\K)$ by Theorem \ref{thm:KtoST} item (\ref{it:onB}). Hence by Lemma \ref{lem:cofinite+}, the closure of $\cQ$  in $\cH^2(\K)$ has finite codimension in $\cH^2(\K)$. Equivalently,  the closure of $\{\sum_{j=1}^ \alpha r_j(S,T)\tilde{q}_j : r_j \in \mathbb{C}[z,w] \}$ is of finite codimension in $\cH^2(\K)$. Thus $(S,T)$ is  $\alpha$-cyclic  on $\bar{\cQ}$  and hence at most nearly $\alpha$-cyclic in $\cH^2(\K)$. 

Moreover, by Corollary  \ref{cor:nearlybound},  $(S,T)$ has rank  at most $\alpha$. For $(\zeta,\eta) \in \Vfp(\fp)$ and for $\gamma \in \C^{\alpha},$ note that 
\[
\K_{(\zeta,\eta)}\gamma \in \ker(M_z-\zeta)^* \cap \ker(M_w-\eta)^*.
\]
 Hence, if $(\zeta,\eta)\in \Vfp(\fp)$ is such that $\K_{(\zeta,\eta)}$ has full rank $\alpha$, then $\ker(M_z-\zeta)^* \cap \ker(M_w-\eta)^*$ has dimension at least $\alpha$. Therefore, $(S,T)$ has rank at least $\alpha$.  Thus $(S,T)$ has rank $\alpha$. By Corollary \ref{cor:nearlybound}, $(S,T)$ is at least nearly $\alpha$-cyclic and hence $(S,T)$ is  nearly $\alpha$-cyclic on $\cH^2(\K)$.
\end{proof}

\begin{proposition}
\label{prop:cyclicvrank}
If $V=(S,T)$ is a finite bimultiplicity pure $\fp$-isopair of rank  $\alpha$ acting on the Hilbert space $\cK$, then there exists a finite codimension $V$ invariant subspace $\cH$ of $\cK$ such that the restriction of $V$ to $\cH$ is $\alpha$-cyclic.
\end{proposition}

\begin{proof}
 Combine Theorems \ref{thm:STtoK} and \ref{S,T nearly alpha cyclic}. 
\end{proof}

\section{DECOMPOSITION OF FINITE RANK ISOPAIRS} \label{sec:5}

\begin{proposition}
	\label{prop:funnycyclic}
	Suppose $p_1,\, p_2\in \C[z,w]$ are relatively prime square free inner toral polynomials, but not necessarily irreducible. If $V_j=(S_j,T_j)$ are  $\beta_j$-cyclic $p_j$-pure isopairs, then $V=V_1\oplus V_2$ is a $p_1p_2$-isopair and is at most nearly  $\max\{\beta_1,\beta_2\}$-cyclic.  
\end{proposition}

\begin{proof}
Clearly,$$p_1p_2(V)=(p_1(V_1)\oplus p_1(V_2))(p_2(V_1)\oplus p_2(V_2)) =(0\oplus p_1(V_2))(p_2(V_1)\oplus 0)=0.$$
Let $I$ be the ideal generated by  $p_1$ and $p_2$. By Proposition \ref{prop:finite-codim} item \eqref{it:finite}, $I$ has finite codimension in $\mathbb{C}[z,w]$. Hence there exists a finite dimension subspace $\mathcal R$ of $\mathbb{C}[z,w]$ such that, for each $\psi\in \C[z,w]$, there exist $s_1,s_2\in \C[z,w]$ and $r\in \mathcal R $ such that
\[
 \psi = s_1 p_1 + s_2 p_2 +r.
\]
 Let $\cK$ denote the Hilbert space that $V$ acts upon. Let $\beta=\max\{\beta_1,\beta_2\}$ and suppose without loss of generality $\beta_1=\beta_2=\beta$.  For $j=1,2$, choose cyclic sets $\Gamma_j =\{\gamma_{j,1}, \dots,\gamma_{j,\beta}\}$ for $V_j$. (In the case where $\beta_1 < \beta _2$ we can set  $\Gamma_1$ to be $\{\gamma_{1,1}, \dots,\gamma_{1,\beta_1}, 0, 0,  ... ,0\}$, so that  this new $\Gamma_1$ has $\beta=\beta_2$ vectors.) Let $\cK_0=\{\psi_1(V_1)\gamma_{1,k} \oplus \psi_2(V_2)\gamma_{2,k}: 1\le k\le \beta, \, \psi_j\in\C[z,w]\}$. By the hypothesis, $\cK_0$ is dense in $\cK$.
For given  polynomials $\psi_1 , \psi_2 \in \C[z,w]$,  there exist $s_1 , s_2 \in \C[z,w]$  and $r \in \mathcal R $ such that
\[
 \psi_1-\psi_2 = - s_1 p_1 + s_2 p_2 + r.
\]
 Rearranging gives,
\[
 \psi_1+s_1  p_1 = \psi_2+s_2  p_2 + r.
\]
 Let $\varphi = \psi_1 +s_1 p_1$. It follows that
\[
 \varphi =  \psi_2 + s_2 p_2 + r.
\]
 Consequently,
\[
 \begin{split} 
   \varphi(V) \left [\gamma_{1,k} \oplus \gamma_{2,k}\right ] = & \varphi(V_1)\gamma_{1,k} \oplus \varphi(V_2)\gamma_{2,k} \\
  = & \psi_1(V_1)\gamma_{1,k} \oplus (\psi_2 (V_2) \gamma_{2,k}  + r(V_2)\gamma_{2,k}).
\end{split}
\]
 Let $\cH_0$ denote the  span of $\{\varphi(V) \left [\gamma_{1,k} \oplus \gamma_{2,k}\right ] : 1 \leq k \leq \beta, \ \psi \in \mathbb{C}[z,w] \}$ and $\cH$ be the closure of $\cH_0$. Let $\mathcal L$ denote the span of $\{0\oplus r(V_2)\gamma_{2,k}:1\le k\le \beta, r\in \mathcal R\}$. Note that $\mathcal L$ is finite dimensional since $\mathcal R$ is and hence $\mathcal L$ is closed. Moreover, 
\[
 \cK_0 = \cH_0 + \mathcal L.
\]
 Hence $\cH_0$ has finite codimension in $\cK_0$. By Lemma \ref{lem:cofinite+}, $\cH$ has finite codimension in $\cK$. Evidently $\cH$ is  $V$ invariant and the restriction of $V$ to $\cH$ is at most  $\beta$-cyclic. Therefore, $V$ is at most nearly $\beta$-cyclic. 
 \end{proof}

\begin{proposition} 
\label{lem:funnycyclic}
  
  If  $V_j=(S_j,T_j)$ are  finite bimultiplicity pure $\fp_j$-isopairs with rank $\alpha_j$ acting on  Hilbert spaces $\cK_j$, where  $\fp_j$ are  irreducible and relatively prime  inner toral polynomials for $1\leq j \leq s$,  then $\bigoplus_{j=1}^s V_j$  is  nearly $\max\{\alpha_1,\alpha_2,...,\alpha_s \}$-cyclic on $\bigoplus_{j=1}^s\cK_j$. 
\end{proposition}

\begin{proof}
 First suppose $s=2$. By Proposition \ref{prop:cyclicvrank}, each $V_j$ is $\alpha_j$-cyclic on  some finite codimensional invariant  subspace $\cH_j$ of $\cK_j$. By Proposition \ref{prop:funnycyclic}, $V_1|_{\cH_1} \oplus V_2|_{\cH_2}$ is at most  nearly $\max\{\alpha_1,\alpha_2\}$-cyclic on $\cH_1 \oplus \cH_2$.  Since each $\cH_j$ has finite codimension in $\cK_j$, it follows that $V=V_1\oplus V_2$ is at most  nearly $\max\{ \alpha_1,\alpha_2\}$-cyclic on $\cK_1 \oplus \cK_2$. On the other hand, $V$ has rank $(\alpha_1,\alpha_2)$ and hence, by Corollary \ref{cor:nearlybound}, is at least $\max\{\alpha_1,\alpha_2\}$-cyclic. Thus $V$ is nearly $\max\{\alpha_1,\alpha_2\}$-cyclic.
 
 Arguing by induction, suppose the result is true for  $0\le j-1<s.$ Thus  $V'= V_1 \oplus ...\oplus V_{j-1}$ is nearly $\beta = \max\{\alpha_1,\alpha_2,...,\alpha_{j-1} \}$-cyclic on $\cK'= \cK_1 \oplus \cK_2 \oplus ... \oplus \cK_{j-1}.$ Hence there exists a finite codimensional invariant  subspace $\cH'$ of $\cK'$ such that the restriction of $V'$ to $\cH'$ is $\beta$-cyclic. Since $V_j$ is a finite bimultiplicity $\fp_j$ isopair with rank $\alpha_j$, by Proposition \ref{prop:cyclicvrank},  there exists a finite codimensional invariant subspace $\cH_j$ of $\cK_j$ such that $V_j|_{\cH_j}$ is $\alpha_j$-cyclic. Note that $\fp_1...\fp_{j-1}$ and $\fp_j$ are relatively prime. Applying Proposition \ref{prop:funnycyclic} to $V'|_{\cH'}$ and $  V_j|_{\cH_j}$, it follows that $V'|_{\cH'} \oplus V_j|_{\cH_j}$ is at most nearly $\gamma=\max\{\beta, \alpha_j\}$-cyclic on $\cH' \oplus \cH_j$. Since $\cH'$ and  $\cH_j$ have finite codimension in $\cK'$ and $\cK_j$ respectively,  $W=V_1\oplus V_2 \oplus ... \oplus V_j$ is at most nearly  $\gamma$-cyclic on $\cK_1 \oplus \cK_2 \oplus \dots \oplus \cK_j$. On the other hand, $W$ has rank $(\alpha_1,...,\alpha_j)$ and  is therefore at least nearly $\gamma$-cyclic by Corollary \ref{cor:nearlybound}. Thus $W$ is nearly $\gamma = \max\{\alpha_1,...,\alpha_j\}$-cyclic.
 \end{proof}
 


\begin{proof}[{Proof of Theorem \ref{thm:main}}]
By \cite[Theorem 2.1]{AKM}, there exist a finite codimension   subspace $\cH$ of $\cK$ that is invariant for $V$ and pure $\fp_j$-isopairs $V_j$ such that
\[
 W = V|_\cH = V_1\oplus V_2 \oplus ...\oplus V_s.
\]
 By Proposition \ref{prop:keeptherank}, $W$ has rank $\alpha$. Hence $V_j$ has rank $\alpha_j$. By Proposition \ref{lem:funnycyclic}, there is a finite codimension invariant subspace $\cL$ of $\cH$ such that the restriction of $W$ to $\cL$ is $\beta=\max\{\alpha_1,\alpha_2,...,\alpha_s\}$-cyclic. Thus $\cL$ is a finite codimensional subspace of $\cK$ such that $V|_{\cL}$ is $\beta$-cyclic. Hence $V$ is at most nearly  $\beta$-cyclic.  By Corollary \ref{cor:nearlybound}, $V$ is at least nearly $\beta$-cyclic. Therefore, $V$ is nearly $\max\{\alpha_1,\alpha_2,...,\alpha_s\}$ cyclic. 
\end{proof}

\begin{Cor}
   Suppose $V=(S,T)$ is a pure $\fp$-isopair of finite bimultiplicity   with minimal polynomial $\fp$ and  write $\fp=\fp_1\,\fp_2\, \cdots \fp_s$ as a product of distinct irreducible factors. If $V$ has  rank $\alpha$ and $\beta =\max\{\alpha_1,...,\alpha_s\}$,  then
 \begin{enumerate}[(i)]
  \item \label{it:1} there exists a finite codimension invariant subspace $\cH$ for $V$ such that the restriction of $V$ to $\cH$ is $\beta$-cyclic;
  \item \label{it:2} $V$ is not $k$-cyclic for any $k<\beta$; and
  \item \label{it:3} there exists a $\beta$-cyclic pure $\fp$-isopair $V^\prime$ and an invariant subspace $\cK$ for $V^\prime$ such that $V$ is the restriction of $V^\prime$ to $\cK$.
\end{enumerate}
\end{Cor}

\begin{proof}
Proofs of items \eqref{it:1} and \eqref{it:2} follow from Theorem \ref{thm:main} and the definition of nearly $k$-cyclic isopairs. The proof of item \eqref{it:3} is an application of item \eqref{it:1} and Corollary \ref{cor:nearlyuniequi}.  
\end{proof}

 \subsection{EXAMPLE}
 In this section we discuss an example on pure $\fp$-isopairs of finite rank to illustrate the connection of the  rank of a pure $\fp$-isopair to  nearly cyclicity and to the representation as direct sums. 
 
 Consider the irreducible, square free  inner toral polynomial,  $\fp =z^3-w^2$.  The distinguished variety, $\mathcal V$,  defines by $\fp$ is called \df{Neil parabola} \cite{JKM}.   The  triple $(\mathcal K_1 ,  Q_1, P_1)$ given by 
 \[
 Q_1(z,w) = \begin{pmatrix} 1 & w 
 \end{pmatrix}\]
 \[
 P_1(z,w) = \begin{pmatrix} 1 & z & z^2
 \end{pmatrix}\]
  and the corresponding kernel function
  \[
  \dfrac{1 +w\overline{\eta}}{1-z\overline{\zeta}} = \mathcal K_1 ((z,w),(\zeta,\eta)) = \dfrac{1+z\overline{\zeta}+z^2 \overline{\zeta}^2}{1-w\overline{\eta}},
  \]
  is a 1-admissible triple. Likewise for the choice of 
  \[
  Q_2(z,w) = \begin{pmatrix} z & w 
  \end{pmatrix}\]
  \[
  P_2(z,w) = \begin{pmatrix} w & z & z^2
  \end{pmatrix}\]
  and the corresponding kernel function
  \[
  \dfrac{z\overline{\zeta} +w\overline{\eta}}{1-z\overline{\zeta}} = \mathcal K_2 ((z,w),(\zeta,\eta)) = \dfrac{w\overline{\eta}+z\overline{\zeta}+z^2 \overline{\zeta}^2}{1-w\overline{\eta}},
  \]
 the triple $(\mathcal K_2,  Q_2, P_2)$ is also a 1-admissible triple. For  $j=1,2$, let $V_j$ be  the pair $ (M_z,M_w)$ defines on $H^2(\K_j)$. Now each  $V_j$ is a  pure $\fp$-isopair or rank 1 and  each $V_j$ is nearly 1-cyclic. 
 
 Let $Q=Q_1 \oplus Q_2$, $P=P_1 \oplus P_2$  and $\K = \K_1 \oplus \K_2$. Observe that the triple $(\K, Q,P)$ is a 2-admissible triple and $V=(M_z,M_w)$ define on $H^2(\K)$  is a pure $\fp$-isopair and nearly 2-cyclic. In fact $V$ is a  pure $\fp$-isopair of rank  2  that can be written  as direct sum of two pure $\fp$-isopairs, $V_1$ and $V_2$.  
 
 However this is not in true general. In other words, there exist  pure $\fp$-isopairs of finite rank (say $\alpha \in \mathbb{N}$),  that cannot be expressed as a direct sum of $\alpha$ number of pure $\fp$-isopairs.  For instance,  let 
 \[
 H' = \displaystyle{\{ f \in H^2(\K) : \langle f, {( 1 \ \  -z )}^{\top} \rangle =0\}}
 \]
 and $V' =  V|_{H'}$. Observe that $H'$ is a finite codimensional subspace of $H^2(\K)$ and $H'$ is invariant under $V$. By the stability of the rank $V'$ has rank 2 and hence nearly 2-cyclic. 
 
 Moreover, the collection of vectors of the form: 
 \[
 \begin{pmatrix} 1/{\sqrt 2} \\ z/ {\sqrt 2} \end{pmatrix} , \begin{pmatrix} z^n \\ 0 \end{pmatrix}_{n \geq 1},  \begin{pmatrix} wz^n \\ 0 \end{pmatrix}_{n \geq 0}, \begin{pmatrix}   0 \\ z^n \end{pmatrix}_{n \geq 2}, \begin{pmatrix} 0 \\ wz^n\end{pmatrix}_{n \geq 0}, 
 \]
forms an orthonormal basis for $H'$.   Hence the reproducing kernel , $\tilde{\K}$, for $H'$ has the form 
\[
\tilde{\mathcal K} ((z,w),(\zeta,\eta)) = \begin{pmatrix}
\dfrac{1}{2}+ \dfrac{z\overline{\zeta} +w\overline{\eta}}{1-z\overline{\zeta}} & \dfrac{\overline{\zeta}}{2}\\
\dfrac{z	}{2} & z\overline{\zeta} \left(\dfrac{1}{2} + \dfrac{z\overline{\zeta} +w\overline{\eta}}{1-z\overline{\zeta}}\right)
\end{pmatrix}.
\]
Since  $\tilde{\mathcal K} ((z,w),(0,0))$ is not diaganalizable, $\tilde{\K}$ is not diagonalizable.  Consequently, $H'$ and $V'$ are not  direct sums. In otherwords,  $V'$ is a pure $\fp$-isopair of rank 2 that cannot be expressed as a direct sum of two other pure $\fp$-isopairs.

\section{ACKNOWLEDGEMENTS}

I am very grateful to my adviser, Scott McCullough, for his valuable guidance and insights that greatly improved the content of this paper.


\end{document}